\RequirePackage[english]{babel} 
\documentclass[a4paper,12pt]{amsart}

\usepackage[T1]{fontenc}
\usepackage[latin1]{inputenc}

\usepackage{url}                    
\usepackage{textcomp}
\usepackage{amssymb}
\usepackage{enumerate}
\usepackage{stmaryrd}
\usepackage{eurosym}
\usepackage[mathcal]{euscript}
\usepackage{bbm}

\theoremstyle{plain}
\newtheorem{teor}{Theorem}[section]
\newtheorem{seur}[teor]{Corollary}
\newtheorem{lemma}[teor]{Lemma}
\theoremstyle{definition}

\theoremstyle{remark}

\numberwithin{equation}{section} 

\DeclareMathOperator{\R}{\mathbbm{R}}
\DeclareMathOperator{\Z}{\mathbbm{Z}}
\DeclareMathOperator{\N}{\mathbbm{N}}

\DeclareMathOperator{\WC}{\mathcal{B}}

\DeclareMathOperator{\diam}{diam}
\DeclareMathOperator{\dist}{dist}

\newcommand{\dr}{d_{\rho}}
\newcommand{\de}{d_{\epsilon}}
\newcommand{\mue}{\mu_{\epsilon}}
\newcommand{\length}{\text{length}}

\newenvironment{proof1}{\textit{Proof of Theorem 1.2.}~}{\hfill$\Box$}

\begin{document}

\title[Gromov hyperbolicity and quasihyperbolic geodesics]{Gromov hyperbolicity and quasihyperbolic geodesics}

\date{\today}
\subjclass[2000]{Primary 30C65.}
\author[P. Koskela, P. Lammi and V. Manojlovi\'c]{Pekka Koskela, P\"aivi Lammi and \\Vesna Manojlovi\'c}
\thanks{The first author was supported  by the Academy of Finland grant 131477.}
\address{Department of Mathematics and Statistics\\
P.O.Box 35 (MaD)\\ 
FI--40014 University of Jyväskylä\\
Finland}
\email{pekka.j.koskela@jyu.fi, paivi.e.lammi@jyu.fi}
\address{University of Belgrade\\ 
Faculty of Organizational Sciences\\
Jove Ilica 154, Belgrade\\ 
Serbia} 
\email{vesnam@fon.bg.ac.rs}
\begin{abstract}
We characterize Gromov hyperbolicity of the quasihyperbolic metric space $(\Omega,k)$ by geometric properties of the Ahlfors regular length metric measure space $(\Omega,d,\mu).$ The characterizing properties are called the Gehring--Hayman condition and the ball--separation condition.
\end{abstract} 

\maketitle

\section{Introduction}

Given a proper subdomain $\Omega$ of the Euclidean space $\R^n,$ $n\ge 2,$
equipped with the usual Euclidean distance, 
one defines the \textit{quasihyperbolic metric} $k$
in $\Omega$ as the path metric generated
by the density  
\[\rho(z)=\frac{1}{d(z)},\] 
where $d(z)=\dist(z,\partial \Omega).$ Precisely, one sets
\[ k(x,y)=
\inf_{\gamma_{xy}}\int_ {\gamma_{xy}}\rho(z)\, ds,
\]
where the infimum is taken over all rectifiable curves $\gamma_{xy}$ that
join $x$ and $y$ in $\Omega$ and the integral is the usual line integral.
Then $\Omega$ equipped with $k$ is a geodesic: there is a curve $\gamma_{xy}$
whose length in the above sense equals $k(x,y).$  Let us denote by $[x,y]$ any
such geodesic; these geodesics are not necessarily unique as can be easily
seen, for example for $\Omega=\R^n\setminus\{0\}.$
The quasihyperbolic metric $k$ was introduced in \cite{GP} and \cite{GO}
where the basic properties of it were established.

If for all triples of geodesics $[x,y],\, [y,z],\, [z,x]$ in $\Omega$ every point in 
$[x,y]$ is within 
$k$--distance $\delta$ from $[y,z] \cup [z,x],$ the space $(\Omega,k)$ is called 
\textit{$\delta$--hyperbolic.} Roughly speaking, this means that 
\textit{geodesic triangles} in $\Omega$ are \textit{$\delta$--thin.} 
Moreover, we say that ($\Omega,k)$ is \textit{Gromov hyperbolic} if it is 
$\delta$--hyperbolic for some $\delta.$ 
The following theorem from \cite{BB} that extends results from \cite{BHK} gives a complete 
characterization of Gromov hyperbolicity of $(\Omega,k).$

\begin{teor}\label{bbgromov}
Let $\Omega\subset\R^n$ be a proper subdomain. Then 
$(\Omega,k)$ is Gromov hyperbolic if and only if $\Omega$ satisfies both a
Gehring--Hayman condition and a ball separation condition.
\end{teor}

Above, the Gehring--Hayman condition means that there is a constant $C_{\text{gh}}\ge 1$ 
such that 
for each pair of points $x, y$ in $\Omega$ and for each quasihyperbolic geodesic $[x,y]$ it 
holds that 
\[\length([x,y])\leq C_{\text{gh}}\length(\gamma_{xy}),\]
where $\gamma_{xy}$ is any other curve joining $x$ to $y$ in $\Omega.$ 
In other words, it says that quasihyperbolic geodesics are essentially the 
shortest curves in $\Omega.$ The other condition, a ball separation condition, 
requires the existence of a constant $C_{\text{bs}} \geq 1$ such that for each pair of 
points $x$ and $y,$ for each quasihyperbolic geodesic $[x,y],$ for every $z\in [x,y],$ 
and for every curve $\gamma_{xy}$ joining $x$ to $y$ it holds that
\[B(z,C_{\text{bs}}d(z))\cap \gamma_{xy} \neq \emptyset.\]

Notice that the three conditions in Theorem \ref{bbgromov}, Gromov hyperbolicity and 
the Gehring--Hayman and the ball separation conditions, are only based on metric concepts. 
It is then natural to ask for an extension of
this characterization to an abstract metric setting. Such an extension was
given in \cite{BB}, relying on an analytic assumption that essentially requires
that the space in question supports a suitable Poincar\'e inequality. This very
same condition, expressed in terms of moduli of curve families \cite{HeiK},
is already in force in \cite{BHK}.

The purpose of this paper is to show that Poincar\'e inequalities are not
critical for geometric characterizations of Gromov hyperbolicity of a 
non--complete metric space, equipped with the quasihyperbolic metric.
Our main result reads as follows.

\begin{teor}\label{main}
Let $Q>1$ and let $(X,d,\mu)$ be a $Q$--regular metric measure space with 
$(X,d)$ a locally compact and annularly quasiconvex length space. Let $\Omega$ be a bounded 
and proper subdomain of $X,$ and let $d_\Omega$ be the inner metric on 
$\Omega$ associated to $d.$ Then
$(\Omega,k)$ is Gromov hyperbolic if and only if $(\Omega,d_\Omega)$ 
satisfies both a Gehring--Hayman condition and a ball separation condition.
\end{teor}

Above, annularly quasiconvexity means that there is a constant 
$\lambda \ge 1$ so that, for any $x\in X$ and all $0<r'<r,$ each pair of points $y,z$ in
$B(x,r)\setminus B(x,r')$ can be joined with a path $\gamma_{yz}$ in
$B(x,\lambda r)\setminus B(x,r'/\lambda)$ such that 
$\length(\gamma_{yz})\le \lambda d(y,z),$
$Q$--regularity requires the existence of a constant $C_{\text{q}}$ so that 
$$r^Q/C_{\text{q}}\le \mu(B(x,r))\le C_{\text{q}}r^Q,$$
for all $r>0$ and all $x\in X,$ and the other concepts are defined analogously 
to the Euclidean
setting described in the beginning of our introduction. See Section 2 for the 
precise definitions. In fact, the assumptions of Theorem \ref{main} can be 
somewhat relaxed, see Section 5.

The main point in Theorem \ref{main} is the necessity of the Gehring-Hayman 
and ball separation conditions; their sufficiency is already given in \cite{BB}.

This paper is organized as follows. Section 2 contains necessary definitions. In Section
3 we give preliminaries related the quasihyperbolic metric and Whitney balls.
Section 4 is devoted to the proof of our main technical estimate, and Section 5 contains
the proof of our main result and some generalizations.

\section{Definitions}

Let $(X,d)$ be a metric space. A \textit{curve} means a continuous map 
$\gamma \colon [a,b] \to X$ from an interval $[a,b]\subset \R$ to $X.$ 
We also denote the image set $\gamma([a,b])$ of $\gamma$ by $\gamma.$ 
The \textit{length} 
$\ell_d(\gamma)$ of $\gamma$ with respect to the metric $d$ is defined as
\[\ell_d(\gamma)=\sup \sum_{i=0}^{m-1}d(\gamma(t_i),\gamma(t_{i+1})),\]
where the supremum is taken over all partitions $a=t_0<t_1<\cdots<t_m=b$ 
of the interval $[a,b].$ If $\ell_d(\gamma)<\infty,$ 
then $\gamma$ is said to be a \textit{rectifiable curve.} 
When the parameter interval is open or half--open, we set  
\[\ell_d(\gamma)=\sup\ell_d(\gamma|_{[c,d]}),\] 
where the supremum is taken over all compact subintervals $[c,d].$ 

When every pair of points in $(X,d)$ can be joined with a rectifiable curve, 
the space $(X,d)$ is called \textit{rectifiably connected.} 
If $\ell_d(\gamma_{xy})= d(x,y)$ for some curve $\gamma_{xy}$ 
joining points $x,y \in X,$ then $\gamma_{xy}$ is said to be a \textit{geodesic.} 
If every pair of points in $(X,d)$ can be joined with a geodesic, 
then $(X,d)$ is called a \textit{geodesic space.} Moreover, a \textit{geodesic ray} in $X$ is an 
isometric image in $(X,d)$ of the interval $[0,\infty).$ Furthermore,
for a rectifiable curve $\gamma$ we define the \textit{arc length} 
$s \colon [a,b]\to [0,\infty)$ along $\gamma$ by
\[s(t)=\ell_d(\gamma|_{[a,t]}).\]

Let $(X,d)$ be a geodesic metric space and let $\delta \geq 0$. 
Denote by $[x,y]$ any geodesic joining two points $x$ and $y$ in $X.$ 
If for all triples of 
geodesics $[x,y],\, [y,z],\, [z,x]$ in $X$ every point in $[x,y]$ is within 
distance $\delta$ from $[y,z] \cup [z,x],$ the space $(X,d)$ is called 
\textit{$\delta$--hyperbolic.} In other words, 
\textit{geodesic triangles} in $X$ are \textit{$\delta$--thin.} 
Moreover, we say that a space is 
\textit{Gromov hyperbolic} if it is $\delta$--hyperbolic for some $\delta.$ 
All Gromov hyperbolic spaces in this paper are assumed to be unbounded.

Next, let $(X,d)$ be a locally compact, rectifiably connected and 
non--complete metric space, and denote by $\overline{X}_d$ its metric 
completion. 
Then the \textit{boundary} 
$\partial_d X \mathrel{\mathop:}=\overline{X}_d\setminus X$ is nonempty. 
We write 
\[d(z)\mathrel{\mathop:}=\dist_d(z,\partial_d X)=\inf\{d(z,x) : x \in \partial_d X\}\] 
for $z\in X.$ Given a real number $D \geq 1,$ a curve $\gamma \colon [a,b]\to X$ is 
called a \textit{$D$--quasiconvex curve} if 
\[\ell_d(\gamma)\leq Dd(\gamma(a),\gamma(b)). \]
If $\gamma$ also satisfies  the \textit{cigar condition}
\[\min\{\ell_d(\gamma|_{[a,t]}),\ell_d(\gamma|_{[t,b]})\}\leq D d(\gamma(t))\]
for every $t\in [a,b],$ the curve is called a \textit{$D$--uniform curve.} 
A metric space $(X,d)$ is called a \textit{$D$--quasiconvex space} or 
\textit{$D$--uniform space} if every pair of points in it can be joined 
with a $D$--quasiconvex curve or a $D$--uniform curve respectively. 

Let $\rho \colon X \to (0,\infty)$ be a continuous function. For each rectifiable curve $\gamma \colon [a,b]\to X$ we define the \textit{$\rho$--length} $\ell_\rho(\gamma)$ of  $\gamma$ by 
\[\ell_\rho(\gamma)=\int_\gamma \rho\, ds = \int_a^b \rho(\gamma(t))\, ds(t).\]
Because $(X,d)$ is rectifiably connected, the density $\rho$ determines a 
metric $\dr$, called a \textit{$\rho$--metric,}
\[\dr(x,y)=\inf_{\gamma_{xy}} \ell_\rho(\gamma_{xy}),\]
where the infimum is taken over all rectifiable curves $\gamma_{xy}$ 
joining $x,y \in X.$  If $\rho\equiv 1,$ then $\ell_\rho(\gamma)=\ell_d(\gamma)$ is the length of 
the curve $\gamma$ 
with respect to the metric $d,$ and the metric $\dr=\ell_d$ is the 
\textit{inner metric associated with $d.$} 
Generally, if the distance between every pair of points in the
metric space is the infimum of the lengths of all curves joining the points, 
then the metric space is called a \textit{length space.}

If we choose \[\rho(z)=\frac{1}{d(z)},\] we obtain the 
\textit{quasihyperbolic metric} in $X.$ 
In this special case, we denote the metric $\dr$ by $k_d$ and the 
quasihyperbolic length of the curve $\gamma$ by $\ell_{k_d}(\gamma).$  
Moreover, $[x,y]_{k_d}$ refers to a quasihyperbolic geodesic joining points 
$x$ and $y$ in $X.$ Because we are dealing with many different metrics, 
the usual metric notations will have an additional subscript that refers to 
the metric in use. For ease of notation, terms which refer to the 
metric $\dr$ will have an additional subscript $\rho$ instead of $\dr.$

We say that $(X,d)$ satisfies a \textit{ball separation condition} if there is 
a constant $C_{\text{bs}}\geq 1$ such that for each pair of points $x,y\in X,$ 
for every quasihyperbolic geodesic $[x,y]_{k_d}\subset X,$ for every 
$z\in [x,y]_{k_d},$ and for every curve $\gamma_{xy}$ joining points $x$ 
and $y,$ it holds that
\begin{equation}
B_d(z,C_{\text{bs}}d(z))\cap \gamma_{xy} \neq \emptyset. \tag*{(BS)}
\end{equation} 
Thus the condition says that the ball $B_d(z,C_{\text{bs}}d(z))$ either 
includes at least one of the endpoints of the quasihyperbolic geodesic or it 
separates the endpoints. This condition was introduced in \cite[\S 7]{BHK}. 
We also say that $(X,d)$ satisfies the \textit{Gehring--Hayman condition} if 
there is a constant $C_{\text{gh}}\geq 1$ such that for every $[x,y]_{k_d}$  
it holds that
\begin{equation}
\ell_d([x,y]_{k_d})\leq C_{\text{gh}} \ell_d(\gamma_{xy}), \tag*{(GH)}
\end{equation}
where $\gamma_{xy}$ is any other curve joining $x$ to $y$ in $X.$

A metric space $(X,d)$ is called \textit{minimally nice} if $(X,d)$ is a 
locally compact, rectifiable connected and non--complete metric space, and 
the identity map from $(X,d)$ to $(X,\ell_d)$ is continuous. If $(X,d)$ is 
minimally nice, then the identity map from $(X,d)$ to $(X,k_d)$ is a 
homeomorphism, and $(X,k_d)$ is complete, proper 
(i.e.\,closed balls are compact) and geodesic (cf.\ \cite[Theorem 2.8]{BHK}). 
Furthermore, we define a proper, geodesic space $(X,k_d)$ to be \textit{$K$--roughly starlike,} $K>0,$ with respect to a base point $w\in X,$ 
if for every point $x\in X$ there exists some geodesic ray emanating from $w$ whose distance to $x$ is at most $K.$

Let $\mu$ be a Borel regular measure on $(X,d)$ with dense support.
We call the density $\rho$ a \textit{conformal density} provided it satisfies both a \textit{Harnack inequality} HI($A$) for some constant $A\geq 1:$
\begin{align}
\frac{1}{A}\leq \frac{\rho(x)}{\rho(y)}\leq A \qquad \text{for all } x,y \in B_d(z,\frac{1}{2}d(z)) \text{ and all } z\in X, \tag*{HI($A$)}
\end{align}
and a \textit{volume growth condition} VG($B$) for some constant $B>0:$
\begin{align}
\mu_{\rho}(B_{\rho}(z,r))\leq Br^Q \qquad \text{for all }  z\in X \text{ and } r>0 \tag*{VG($B$)}.
\end{align}
Here $\mu_{\rho}$ is the Borel measure on $X$ defined by
\begin{align*}
\mu_{\rho}(E)=\int_E \rho^Q\, d\mu \qquad \text{for a Borel set } E\subset X,
\end{align*}
and $Q$ is a positive real number. Generally $Q$ will be the Hausdorff 
dimension of our space $(X,d).$ There is nothing special 
about the constant $\frac12$ in condition HI($A$): we may replace it by any 
constant $0 < c \le \frac12,$ actually by any constant $c\in (0,1).$ 
Suppose that we have fixed $0< c \le \frac12.$ Then each ball $B_d(z,c d(z))$ 
is called a \textit{Whitney type ball.} 

In general, we say that $(X,d,\mu)$ is \textit{$Q$--upper regular} for some 
$Q>0$ if there is a constat $C_\text{u} \geq 1$ such that
\begin{equation}\label{upper regular}
\mu(B_d(z,r)) \leq C_\text{u} r^Q
\end{equation}
for each $z \in X$ and every $r>0.$ We also say that $(X,d,\mu)$ is 
\textit{$Q$--regular on Whitney type balls} for some $Q>0$ if there are constants
$C_\text{w}\geq 1$ and $0<\epsilon \le 1$ such that
\begin{equation} \label{locally regular}
C_\text{w}^{-1}r^Q\leq\mu(B_d(z,r))\leq C_\text{w} r^Q
\end{equation}
for each $z\in X$ and every $r \leq \epsilon d(z)/2.$ 

\section{The metric measure space $(X,\de,\mue)$}

Let $(X,d,\mu)$ be a minimally nice metric measure space so that the measure 
$\mu$ is Borel regular and $(X,k_d)$ is Gromov hyperbolic. Let $w\in X$ be a 
base point. We define two deformations by setting 
\[\rho_\epsilon(z)=\exp\{-\epsilon k_d(w,z)\} \quad \text{ and } \quad \sigma_\epsilon(z)=\frac{\rho_\epsilon(z)}{d(z)}\] 
for given $\epsilon >0.$ This generates a metric space $(X,\de)$ with
\[\de(x,y)= \inf_{\gamma_{xy}} \ell_\epsilon (\gamma_{xy})=\inf_{\gamma_{xy}} \int_{\gamma_{xy}} \rho_\epsilon(z)\, d_{k_d}s=\inf_{\gamma_{xy}}\int_{\gamma_{xy}} \sigma_{\epsilon}(z)\, ds, \]
where the infimum is taken over all curves $\gamma_{xy}$ joining points $x$ and $y$ in $X,$ 
and $d_{k_d}s=ds/d(z)$ denotes the quasihyperbolic distance element. 
For ease of notation we write $d_\epsilon$ instead of $d_{\sigma_\epsilon}.$ 
We also refer to the metric $d_\epsilon$ via an additional subscript $\epsilon,$ for instance 
$\overline{X}_\epsilon$ denotes the metric completion of $(X,\de).$ If $Q>0$ is fixed, we also attach the Borel measure $d\mue(z)=\sigma_\epsilon(z)^Qd\mu(z)$ to $(X,d_\epsilon).$

Because $(X,k_d)$ is geodesic, for given $x\in X$ and a geodesic $[w,x]_{k_d}$ we have that
\[d_\epsilon(w,x)\leq \int_{[w,x]_{k_d}}\rho_\epsilon\, d_{k_d}s \leq \int_0^\infty e^{-\epsilon t}\, dt = \frac{1}{\epsilon}.\]
Thus $(X,d_\epsilon)$ is always bounded.
Moreover, by the triangle inequality the density $\rho_\epsilon$ satisfies a  
Harnack type inequality:
\begin{equation}
\exp\{-\epsilon k_d(x,y)\} \leq \frac{\rho_\epsilon(x)}{\rho_\epsilon(y)}\leq \exp\{\epsilon k_d(x,y)\}
\end{equation}
for all $x,y \in X$ and all $\epsilon >0.$  We also obtain that the density $\sigma_\epsilon$ satisfies the Harnack inequality HI($A$) with the constant $A=3\exp\{2\epsilon\}:$
\begin{equation*}
\begin{split}
\frac{1}{3}\exp\{-2\epsilon\}\leq \frac{1}{3}\exp\{-\epsilon k_d(x,y)\}&\leq \frac{\sigma_\epsilon(x)}{\sigma_\epsilon(y)}\\
&\leq 3\exp\{\epsilon k_d(x,y)\}\leq 3\exp\{2\epsilon\} 
\end{split}
\end{equation*}
for all $x,y \in B_d(z,\frac{1}{2}d(z))$ and for every $z\in X.$

Bonk, Heinonen and Koskela proved in \cite[\S 4 and Theorem 5.1]{BHK} that there is 
$\epsilon_0 >0$ depending on $\delta$ such that the metric space $(X,d_\epsilon)$ is 
$D_\epsilon$--uniform for every $0<\epsilon \leq \epsilon_0,$ 
where $k_d$--quasihyperbolic geodesics serve as $D_\epsilon$--uniform curves with $D_\epsilon=D(\delta,\epsilon, \epsilon_0) \geq 1.$ 
Especially, we have a version of the \textit{Gehring--Hayman condition}: there is a constant $D_\epsilon \geq 1$ such that when $\epsilon \leq \epsilon_0,$
\begin{equation}\label{G-H}
\ell_\epsilon([x,y]_{k_d})\leq D_\epsilon\ell_\epsilon(\gamma_{xy})
\end{equation}
for each quasihyperbolic geodesic $[x,y]_{k_d}$ in $X$ and for each curve $\gamma_{xy}$ joining $x$ to $y$ in $X.$ Furthermore, if $(X,k_d)$ is $K$--roughly starlike with respect to the base point $w,$ then by \cite[Lemma 4.17]{BHK} we have that
\begin{equation}\label{distance vs rho}
\begin{split}
\frac{1}{\epsilon e}\sigma_\epsilon(x)d(x)&=\frac{1}{\epsilon e}\rho_\epsilon(x)\leq \de(x)\\
&\leq \frac{2\exp\{\epsilon K\}-1}{\epsilon}\rho_\epsilon(x)\\
&=\frac{2\exp\{\epsilon K\}-1}{\epsilon}\sigma_\epsilon(x)d(x)
\end{split}
\end{equation}
for all $\epsilon >0$ and every $x\in X.$
Thus there exists $c=c(\delta,K) \in (0,1)$ such that
\begin{equation}\label{quasihypmetrics}
c\epsilon k_d(x,y)\leq k_\epsilon(x,y)\leq e\epsilon k_d(x,y),
\end{equation}
for all $x,y \in X$ and every $0<\epsilon \leq \epsilon_0,$ where $k_\epsilon$ is the 
quasihyperbolic metric derived from $\de.$ Moreover, we obtain from 
\cite[Theorem 6.39]{BHK} that Whitney type balls in $(X,\de)$ are also Whitney type balls 
in $(X,d).$ To be more specific, let 
\begin{equation}\label{constant C}
C=\max\Bigl\{3\exp\{2\epsilon_0\},\epsilon_0e,\frac{2\exp\{\epsilon_0K\}-1}{\epsilon_0}\Bigr\}.
\end{equation}
 Then 
\begin{equation}\label{WB1}
B_\epsilon(z,\epsilon\de(z))\subset B_d(z,\epsilon C^2d(z)) \tag*{(WB1)}
\end{equation}
whenever $z\in X$ and $0<\epsilon\leq\min\{\epsilon_0,\frac{1}{2C^2}\}.$ 
Furthermore, if $(X,d)$ is a $D$--quasiconvex space, then Whitney type balls in $(X,d)$ 
are also Whitney type balls in $(X,\de):$
\begin{equation}\label{WB2}
B_d(z,\epsilon d(z))\subset B_\epsilon(z,\epsilon D C^2\de(z)) \tag*{(WB2)}
\end{equation}
whenever $z\in X$ and $0 < \epsilon \leq \min\{\epsilon_0,\frac{1}{8D}\}.$

Moreover, if $(X,d,\mu)$ is $Q$--regular on Whitney type balls and $Q$--upper regular, 
$(X,\de,\mue)$ is $Q$--regular on Whitney type balls, 
when $\epsilon \leq \min\{\epsilon_0,\frac{1}{2C^2}\}.$ 
Indeed, let $z\in X$ and let $r\leq \epsilon \de(z)\leq \frac{1}{2}\de(z).$ 
From \ref{WB1} we obtain that 
\begin{equation}\label{balls}
B_\epsilon(z,r) \subset B_\epsilon(z,\epsilon \de(z))\subset B_d(z,\epsilon C^2d(z)) \subset B_d(z,d(z)/2),
\end{equation}
and by HI($A$) it follows that 
$B_\epsilon(z,r)\subset B_d\Bigl(z,\frac{A}{\sigma_\epsilon(z)}r\Bigr).$ Now, because $\mu$ is $Q$--upper regular, with HI($A$) it follows that
\begin{equation}\label{upper bound}
\begin{split}
\mu_\epsilon(B_\epsilon(z,r))&= \int_{B_\epsilon(z,r)}(\sigma_\epsilon(u))^Q\, d\mu(u) \\
&\leq (A\sigma_\epsilon(z))^Q\int_{B_\epsilon(z,r)}d\mu(u)\\
&\leq (A\sigma_\epsilon(z))^Q\mu\Bigl(B_d\Bigl(z,\frac{A}{\sigma_\epsilon(z)}r\Bigr)\Bigr)\\
&\leq A^{2Q}C_\text{u}r^Q.
\end{split}
\end{equation}
The lower bound follows similarly: When $(X,d)$ is $D$--quasiconvex, by HI($A$) we have that $B_d(z,\frac{1}{A D\sigma_\epsilon(z)}r)\subset B_\epsilon(z,r).$ Thus HI($A$) together with 
\eqref{locally regular} yields 
\begin{equation}\label{lower bound}
\begin{split}
\mu_\epsilon(B_\epsilon(z,r))& = \int_{B_\epsilon(z,r)}(\sigma_\epsilon(u))^Q\, d\mu(u)\\
&\geq \Bigl(\frac{\sigma_\epsilon(z)}{A}\Bigr)^Q\int_{B_\epsilon(z,r)}d\mu(u)\\
&\geq \Bigl(\frac{\sigma_\epsilon(z)}{A}\Bigr)^Q\mu\Bigl(B_d\Bigl(z,\frac{1}{A D\sigma_\epsilon(z)}r\Bigr)\Bigr)\\
&\geq \Bigl(\frac{1}{A}\Bigr)^{2Q}\frac{1}{D^QC_\text{w}}r^Q.
\end{split}
\end{equation}

Let $Q>1$ and let $(X,d,\mu)$ be a minimally nice $D$--quasiconvex and $Q$--upper 
regular space so that the measure $\mu$ is $Q$--regular on Whitney type balls, and 
$(X,k_d)$ is a $K$--roughly starlike Gromov hyperbolic space. 
Let the constants $\epsilon_0$ and $C$ be as in the paragraph containing
\eqref{G-H} and \eqref{constant C}.
We may define a Whitney covering of the space $(X,\de,\mue)$ 
when $\epsilon\leq \min\{\epsilon_0,\frac{1}{8D},\frac{1}{2C^2}\},$ see e.g.\ 
\cite[Theorem III.1.3]{CW}, \cite[Lemma 2.9]{MaSe}, \cite[Lemma 7]{HKT} and \cite[\S 3]{KL}.

Let $r(z)=\epsilon \de(z)/50.$ From the family $\{B_\epsilon(z,r(z))\}_{z\in X}$ of balls we 
select a maximal (countable) subfamily $\{B_\epsilon(z_i,r(z_i)/5)\}_{i\in I}$ of pairwise 
disjoint balls. We write $\WC = \{B_i\}_{i\in I},$ where $B_i=B_\epsilon(z_i,r_i)$ 
and $r_i=r(z_i).$ We call the family $\WC$ the \textit{Whitney covering} of $(X,\de).$ 
We list the basic properties of the Whitney covering in Lemma \ref{properties}. As in 
\cite[Lemma 3.2]{KL}, the property \eqref{the property} is a consequence of the 
$Q$--regularity condition of $\mue$ on Whitney type balls, and the property (v) follows 
from the proof of Lemma 3.2 in \cite{KL} via uniformity and $Q$--regularity condition of $\mue$ on Whitney type balls.

\begin{lemma}\label{properties}
There is $N \in \N$ such that
\renewcommand{\theenumi}{\roman{enumi}}
\renewcommand{\labelenumi}{\textnormal{(\theenumi)}}
\begin{enumerate}
\item the balls $B_\epsilon(z_i,r_i/5)$ are pairwise disjoint,\\
\item $X = \bigcup_{i\in I}B_\epsilon(z_i,r_i),$\\
\item $B_\epsilon(z_i,5r_i)\subset X,$\\
\item $\sum_{i=1}^\infty \chi_{B_\epsilon(z_i,5r_i)}(x)\leq N$ for all $x\in X.$ \label{the property}
\end{enumerate}

Furthermore, suppose $\de(x,y)\geq \de(x)/2$ and let $\gamma_{xy}$ be a $D_\epsilon$--uniform 
curve joining $x$ to $y.$ There exists a constant $C_{\epsilon}>0,$ that depends 
quantitatively on $\epsilon$ and the hypotheses, 
such that
\begin{enumerate}
\setcounter{enumi}{4}
\item $C_{\epsilon}^{-1}N_{\epsilon}(x,y)\leq \ell_\epsilon(\gamma_{xy})\leq C_{\epsilon}N_\epsilon(x,y),$ \label{original number of balls}
\end{enumerate}
where $N_\epsilon(x,y)$ is the number of balls $B\in \WC$ intersecting $\gamma_{xy}.$
\end{lemma}

Fix a ball $B_0$ from the Whitney covering $\WC$ and let $z_0$ be its center.  
For each $B_i\in \WC$ we fix a geodesic $[z_0,z_i]_{k_\epsilon}.$ Furthermore, for each $B_i \in \WC$ we set $P(B_i)=\{B \in \WC : B\cap [z_0,z_i]_{k_\epsilon}\neq \emptyset\}$ and define the \textit{shadow} $S(B)$ of a ball $B \in \WC$ by 
\[S(B)=\underset{B\in P(B_i)}{\bigcup_{B_i\in \WC}}B_i.\]
For $k\in \N$ we set
\[\WC_k=\{B_i \in \WC: k\leq k_\epsilon(z_0,z_i)<k+1\}.\]
Similarly as in \cite[Lemma 3.3 and Lemma 3.4]{KL} we may prove, when $\epsilon\leq \min\{\epsilon_0,\frac{1}{8D}, \frac{1}{2C^2}\},$ that there is a constant $C_\text{o}>0$ 
that depends quantitatively on $\epsilon$ and the hypotheses, such that
\begin{equation}\label{shadow overlap}
\sum_{B\in\WC_k}\chi_{S(B)}(x)\leq C_\text{o}.
\end{equation}

\section{The main lemma}

Next we prove a lemma which is the central tool for proving our main theorem. 
From now on we assume that $Q>1,$ $(X,d,\mu)$ is a minimally nice $Q$--upper 
regular $D$--quasiconvex metric measure space such that the measure $\mu$ is $Q$--regular on 
Whitney type balls, and $(X,k_d)$ is a $K$-roughly starlike Gromov hyperbolic space. 
We also assume that $(X,\de,\mue)$ is a deformation of $(X,d,\mu)$ as described above, 
where $w\in X$ is a base point, $\epsilon_0 >0$ is as in the paragraph containing
\eqref{G-H}, $C>1$ as in 
\eqref{constant C} and $0< \epsilon \leq \min\{\epsilon_0, \frac{1}{8D},\frac{1}{2C^2}\}.$

\begin{lemma}\label{main lemma}
Let $u \in X$ be a point and $\gamma \subset X$ be a curve such that
\[\dist_\epsilon(u,\gamma)\leq \min\{C_1\de(u),C_2\diam_\epsilon(\gamma)\}\]
for some $C_1, C_2 >0.$ 
Then there exists a constant $M\ge 1$ that depends quantitatively on $\epsilon$ and the 
constants in our hypotheses, so that
\[\dist_d(u,\gamma)\leq Md(u).\]
\end{lemma}

\begin{proof}
Suppose that for a fixed $M> 2$ there is $u\in X$ and a curve $\gamma$ such that 
$\dist_\epsilon(u,\gamma)\leq \min\{C_1\de(u),C_2\diam_\epsilon(\gamma)\}$ but  
$\dist_d(u,\gamma) > Md(u).$ We will show that such an $M$ has an upper bound in terms
of our data. Towards this end, let us fix $u$ and $\gamma$ as above. By replacing $\gamma$
with a suitable subcurve of $\gamma$ we may assume without loss of generality that 
$$\gamma\subset B_{\epsilon}(u,2Cd_{\epsilon}(u)),$$
where $C=\max\{C_1,C_1/C_2\}.$

Let $\WC$ be a Whitney covering of $(X,\de,\mue)$ as in Section 3.
We choose $B_\epsilon(u,r(u))$ as the fixed ball $B_0\in \WC.$
Let $\hat{u}\in\partial_dX$ be such that $d(u)=d(u,\hat{u}).$
Let $y\in \gamma$ and let $[u,y]_{k_d}$ be a $k_d$--quasihyperbolic geodesic joining $u$ to 
$y.$ Moreover, let $I_y, J_y\subset \N$ be index sets so that 
\[I_y=\{i\in \N : B_i\in\WC,\, [u,y]_{k_d}\cap 5B_i \cap B_d(u,Md(u)) \neq \emptyset\}\] 
and 
\[J_y=\{j\in \N : B_j\in\WC,\, [u,y]_{k_d}\cap B_j \cap B_d(u,Md(u)) \neq \emptyset\}.\] 
Let $\tilde{y}\in [u,y]_{k_d}$ be the first point in $[u,y]_{k_d}$ such that 
$\tilde{y}\notin B_d(u,Md(u)).$ Thus $[u,\tilde{y}]_{k_d}$ is a subcurve of 
$[u,y]_{k_d}$ in $\overline B_d(u,Md(u)).$  
Let $[u_i,y_i]_{k_d}\subset [u,y]_{k_d},$ $i\in I_y,$ be a subcurve of $[u,y]_{k_d}$ 
such that 
$u_i$ is the first point and $y_i$ is the last point, where $[u,y]_{k_d}$ intersects 
$\overline 5B_i.$ Moreover, let 
$\alpha_j \subset [u_j,y_j]_{k_d} \cap \overline 5B_j,$ $j\in J_y,$ be the 
maximal connected component of $[u,y]_{k_d}$ in $\overline 5B_j$ such that 
$\alpha_j\cap \overline B_j \neq \emptyset.$

We define $g\colon X\to (0,\infty)$ by setting
\[g(x)=\sum_{B \in \WC}\frac{\diam_d(B)}{\log(M-1)\dist_d(B,\hat{u})\diam_\epsilon(B)}\chi_{5B}(x).\]
Because $k_d$--quasihyperbolic geodesics are $D_\epsilon$--uniform curves in $(X,\de),$ 
we obtain that
\begin{align}\label{g-length1}
\int_{[u,\tilde{y}]_{k_d}}g\, d_\epsilon s &\le \frac{1}{\log(M-1)}\sum_{i \in I_y}\frac{\diam_d(B_i)}{\dist_d(B_i,\hat{u})\diam_\epsilon(B_i)}\ell_\epsilon([u_i,y_i]_{k_d})\notag\\  
&\le \frac{5D_\epsilon}{\log(M-1)}\sum_{i \in I_y}\frac{\diam_d(B_i)}{\dist_d(B_i,\hat{u})}.
\end{align}
Here $d_\epsilon s=\sigma_\epsilon (z)\, ds$ is the length element in the metric $\de.$ 

Moreover, by \ref{WB2} and Lemma \ref{properties} \eqref{the property} we obtain that
\begin{equation}\label{g-length2}
\begin{split}
\int_{[u,\tilde{y}]_{k_d}}g\, d_\epsilon s &\geq \frac{1}{N\log(M-1)}\sum_{j\in J_y}\frac{\diam_d(B_j)}{\dist_d(B_j,\hat{u})\diam_\epsilon(B_j)}\ell_\epsilon(\alpha_j)\\
&\geq \frac{2}{N\log(M-1)}\sum_{j\in J_y} \frac{\diam_d(B_j)}{\dist_d(B_j,\hat{u})}\\ 
&\geq \frac{2\epsilon}{25C^2DN\log(M-1)} \sum_{j\in J_y}\frac{d(z_j)}{\dist_d(B_j,\hat{u})}.
\end{split}
\end{equation}

Let $j\in J_y.$ Because $\diam_{k_\epsilon}(B_j)\leq \frac{2\epsilon}{50-\epsilon},$ 
using \ref{WB1} and \eqref{quasihypmetrics}, we obtain the estimate 
\begin{equation}
\begin{split}
\int_{[u,y]_{k_d}\cap B_j} ds &= \int_{[u,y]_{k_d}\cap B_j} \frac{d(z_j)}{d(z_j)}\, ds\\ 
&\leq \frac{50+\epsilon C^2}{50}d(z_j)\diam_{k_d}(B_j)\\ 
&\leq \frac{50+\epsilon C^2}{50}d(z_j)\frac{1}{c\epsilon}\diam_{k_\epsilon}(B_j) \\
&\leq \frac{50+\epsilon C^2}{25c}\frac{1}{50-\epsilon}d(z_j).
\end{split}
\end{equation}
Combining this with \eqref{g-length2} and \eqref{g-length1} we have that
\begin{align}\label{g-length final}
\frac{5D_\epsilon}{\log(M-1)}&\sum_{i \in I_y}\frac{\diam_d(B_i)}{\dist_d(B_i,\hat{u})}\notag\\
&\geq \frac{2c\epsilon}{C^2DN\log(M-1)}\frac{50-\epsilon}{50+\epsilon C^2} \sum_{j\in J_y} \int_{[u,y]_{k_d}\cap B_j} \frac{1}{\dist_d(B_j,\hat{u})}\, ds\notag\\
&\geq\frac{2c\epsilon}{C^2DN\log(M-1)}\frac{50-\epsilon}{50+\epsilon C^2}\int_{[u,\tilde{y}]_{k_d}} \frac{1}{d(z,\hat{u})}\, ds\\
&\geq\frac{2c\epsilon}{C^2DN\log(M-1)}\frac{50-\epsilon}{50+\epsilon C^2}\log \Bigl(\frac{d(\tilde{y},\hat{u})}{d(u,\hat{u})}\Bigr) \notag\\
&\geq \frac{2c\epsilon}{C^2DN}\frac{50-\epsilon}{50+\epsilon C^2}.\notag
\end{align}
Now, from \eqref{g-length final} we obtain the estimate
\begin{equation}\label{estimate for one}
1 \le 
 \frac{P}{\log(M-1)}\sum_{i\in I_y}\frac{\diam_d(B_i)}{\dist_d(B_i,\hat{u})},
\end{equation}
where $P \geq 1$ is a constant that depends on $\epsilon$ and the constants of the hypotheses.

Let $\nu$ be a Radon measure on $\gamma \subset (X,\de)$ given by Frostman's lemma 
(cf.\ \cite{Ma} and \cite[Theorem 4.1]{KL}) so that
\begin{equation}\label{Frostman}
\begin{split}
\nu(E)&\leq \diam_\epsilon (E) \quad \text{ for every } E\subset \gamma, \text{ and} \\
\nu(\gamma)&\geq \frac{1}{30}\frac{\diam_\epsilon (\gamma)}{2}.
\end{split}
\end{equation}
For every $y\in \gamma$ we set 
\[S_y=\{x\in \gamma : x\in \bigcup_{i\in I_y}S(B_i)\}.\] 
We may choose a finite number of points $y_n\in \gamma$ such that 
$\gamma \subset \bigcup_n S_{y_n}.$ Hence, using \eqref{estimate for one}, 
 Fubini's theorem and H\"older's inequality we obtain that
\begin{equation}\label{start}
\begin{split}
\nu(\gamma)&=\int_\gamma d\nu \leq \frac{P}{\log(M-1)}\int_\gamma \sum_{i \in I_y}\frac{\diam_d(B_i)}{\dist_d(B_i,\hat{u})}\, d\nu \\
&\leq \frac{P}{\log(M-1)}\sum_{n} \sum_{i\in I_{y_n}}\frac{\diam_d(B_i)}{\dist_d(B_i,\hat{u})}\nu(S(B_i)\cap \gamma)\\
&\leq \frac{P}{\log(M-1)}\sum_{k=0}^\infty \underset{i\in I_y,\, y\in \gamma}{\sum_{B_i\in\WC_k}}\frac{\diam_d(B_i)}{\dist_d(B_i,\hat{u})}\nu(S(B_i)\cap \gamma)\\
&\leq \frac{P}{\log(M-1)}\Bigl(\sum_{k=0}^\infty \underset{i\in I_y,\, y\in \gamma}{\sum_{B_i\in \WC_k}}\Bigl(\frac{\diam_d(B_i)}{\dist_d(B_i,\hat{u})}\Bigr)^Q\Bigr)^{\frac{1}{Q}} \\
& \qquad \qquad \qquad \quad \Bigl(\sum_{k= 0}^\infty \underset{i\in I_y,\, y\in \gamma}{\sum_{B_i\in \WC_k}}(\nu(S(B_i)\cap \gamma))^{\frac{Q}{Q-1}}\Bigr)^{\frac{Q-1}{Q}}. 
\end{split}
\end{equation}

Let us first estimate the first double sum in \eqref{start}. 
Let 
\[A_n=\Bigl(\overline{B}_d(\hat{u},2^nd(u))\setminus B_d(\hat{u},2^{n-1}d(u))\Bigr) \cap X.\] 
Pick an integer $m$ with $2^{m-1}< M+1 \leq 2^m.$ 
Write $\tilde{B}=B_d(z,\frac{\epsilon C^2}{50}d(z))$ for $B=B_{\epsilon}(z,r(z))\in \WC.$ 
Given $B\in \WC,$ \ref{WB1} ensures that $B\subset \tilde{B}.$  
Moreover, if $B\cap A_n\neq \emptyset$ for some $n\in \Z,$ 
then $\tilde{B}\subset A_{n-1}\cup A_n\cup A_{n+1}.$ Thus, 
by the $Q$--upper regularity condition, the $Q$--regularity condition on 
Whitney type balls in $(X,d)$ and Lemma \ref{properties} \eqref{the property} we deduce that 
\begin{align}\label{sums1}
\sum_{k=0}^\infty \underset{i\in I_y,\, y\in \gamma}{\sum_{B_i\in \WC_k}}\Bigl(\frac{\diam_d(B_i)}{\dist_d(B_i,\hat{u})}\Bigr)^Q &\leq \sum_{n=0}^{m} \underset{B\cap A_n \neq \emptyset}{\sum_{B\in \WC}}\Bigl(\frac{\diam_d(B)}{\dist_d(B,\hat{u})}\Bigr)^Q \notag\\
&\leq \sum_{n=0}^{m} \underset{B\cap A_n \neq \emptyset}{\sum_{B\in \WC}}\frac{2^QC_\text{w}\mu(\tilde{B})}{(\dist_d(B,\hat{u}))^Q}\notag\\
&\leq 2^QC_\text{w}^3(5 DC^4)^Q\sum_{n=0}^m\frac{\mu(B_d(\hat{u},2^{n+1}d(u))\cap X)}{(2^{n-2}d(u))^Q}\\
&\leq 2^{4Q}C_\text{w}^3C_\text{u}(5 DC^4)^Q (m+1)\notag\\
&\leq 2^{4(Q+1)}C_\text{w}^3C_\text{u}(5 DC^4)^Q \log M.\notag 
\end{align}

Above, in moving from the second line to the third, we used the pairwise disjointness of the
balls $\frac 1 5 B,$ (WB2) and the $Q$--regularity of $\mu$ on Whitney type balls.

Let us then estimate the second double sum in \eqref{start}.
Let $B_i\in \WC,$ where $i \in I_y$ with $y\in \gamma,$ and let $z_i\in B_i$ be its 
center. 
Because $k_d$--quasihyperbolic geodesics are $D_{\epsilon}$--uniform curves and
$\dist_\epsilon(u,\gamma)\leq \min\{C_1\de(u),C_2\diam_\epsilon(\gamma)\},$ 
we conclude that 
\begin{align*}
\de(z_i)&\leq \de(u)+D_\epsilon \dist_\epsilon(u,\gamma) + \frac{\epsilon}{10} \de(z_i)\\
& \leq (1+D_\epsilon C_1)\de(u) + \frac{\epsilon}{10}\de(z_i),
\end{align*}
and thus 
\begin{align}\label{distance1}
\de(z_i) \leq \frac{10}{10-\epsilon}(1+D_\epsilon C_1)\de(u).
\end{align}
We also know that $B_\epsilon(u,\epsilon \de(u))\subset B_d(u,\frac 1 2 d(u))
\subset B_d(u,Md(u))$  and $\gamma\cap B_d(u,Md(u))=\emptyset.$  Especially
$\gamma\cap B_{\epsilon}(u,\epsilon d_{\epsilon}(u))=\emptyset$ and hence
\begin{align*}
\de(z_i) &\leq \de(u) + D_\epsilon \dist_\epsilon (u,\gamma) + \frac{\epsilon}{10}\de(z_i)\\
& \leq \frac{1}{\epsilon}\dist_\epsilon (u,\gamma) + D_\epsilon \dist_\epsilon(u,\gamma) + \frac{\epsilon}{10}\de(z_i).
\end{align*}
Thus
\begin{align}\label{distance2}
\de(z_i)\leq \frac{10}{10-\epsilon}\frac{1+\epsilon D_\epsilon }{\epsilon} C_2 \diam_\epsilon(\gamma).
\end{align}
If $B_i\in \WC_k,$ 
since $(X,\de)$ is a $D_\epsilon$--uniform space, by \cite[Lemma 2.13]{BHK} we have that 
\begin{equation}\label{distance 3}
k\leq k_\epsilon(u,z_i)\leq 4D_\epsilon^2\log\Bigl(1+\frac{\de(u,z_i)}{\min\{\de(u),\de(z_i)\}}\Bigr).
\end{equation}
Moreover, since $k_d$--quasihyperbolic geodesics are $D_\epsilon$--uniform curves
in $(X,d_\epsilon)$, \eqref{distance2} implies that
\begin{align*}
\de(u,z_i)&\leq D_\epsilon\dist_\epsilon(u,\gamma) + \frac{\epsilon}{10}\de(z_i)\\
&\leq \Bigl(D_\epsilon+\frac{1+\epsilon D_\epsilon}{10-\epsilon}\Bigr)C_2 \diam_\epsilon(\gamma),
\end{align*} 
and thus \eqref{distance1}, \eqref{distance 3} and this give us when $k \ge 4D_\epsilon ^2$ that
\begin{equation} \label{estimate}
\begin{split}
\de(z_i) &\leq \frac{10}{10-\epsilon}(1+D_\epsilon C_1)\min\{\de(u),\de(z_i)\}\\
&\leq \frac{10}{10-\epsilon}(1+D_\epsilon C_1)\frac{\de(u,z_i)}{\exp\{\frac{k}{4D_\epsilon^2}\}-1}\\
&\le 2^{\frac{-k}{4D_\epsilon^2}}\frac{2(10+10D_\epsilon C_1)(10D_\epsilon +1)}{(10-\epsilon)^2}C_2 \diam_\epsilon(\gamma).
\end{split}
\end{equation}

Because $(X,\de)$ is a $D_\epsilon$--uniform space, and 
$\gamma\subset B_{\epsilon}(u,2Cd_{\epsilon}(u)),$ it easily follows that there is a constant
$C_{\text{s}}$ only depending on $D_\epsilon$ and $C$ so that
\begin{equation}\label{size of the shadow}
\diam_\epsilon(S(B_i)\cap \gamma)\leq C_\text{s} \diam_\epsilon(B_i),
\end{equation}
for each $y\in \gamma$ and every $i\in I_y.$
Now, inequalities \eqref{shadow overlap}, \eqref{Frostman}, \eqref{size of the shadow}, 
\eqref{distance2} and \eqref{estimate} yield
\begin{align}\label{sums2}
\sum_{k=0}^\infty \underset{i\in I_y, y\in \gamma}{\sum_{B_i\in \WC_k}}(\nu(S(B_i)\cap \gamma))^{\frac{Q}{Q-1}}& 
\leq \sum_{k=0}^\infty \underset{i\in I_y, y\in\gamma}{\max_{B_i\in \WC_k}}(\nu(S(B_i)\cap\gamma))^{\frac{1}{Q-1}} \underset{i\in I_y, y\in \gamma}{\sum_{B_i\in \WC_k}}\nu(S(B_i)\cap \gamma)\notag\\
&\leq C_\text{o}\nu(\gamma)\sum_{k=0}^\infty ( \underset{i\in I_y, y\in \gamma}{\max_{B_i\in \WC_k}}(\diam_\epsilon(S(B_i)\cap\gamma)))^{\frac{1}{Q-1}}\notag\\
&\leq C_\text{o}\Bigl(\frac{\epsilon C_\text{s}}{25}\Bigr)^{\frac{1}{Q-1}}\nu(\gamma)\sum_{k=0}^\infty( \underset{i\in I_y, y \in \gamma}{\max_{B_i\in \WC_k}}(\de(z_i)))^{\frac{1}{Q-1}}\\
&\leq C' \nu(\gamma)(\diam_\epsilon (\gamma))^{\frac{1}{Q-1}},\notag
\end{align}
where $C' >0$ is a constant depenging on $\epsilon$ and the hypotheses.

Combining \eqref{sums1} and \eqref{sums2} with \eqref{start} we obtain
\begin{equation*}
\nu(\gamma)\leq C'' \frac{(\log M)^{\frac{1}{Q}}}{\log(M-1)}\nu(\gamma)^{\frac{Q-1}{Q}}(\diam_\epsilon(\gamma))^{\frac{1}{Q}}. 
\end{equation*}
Inserting \eqref{Frostman} we conclude that
\begin{equation*}
1\leq 60 (C'')^Q  \frac{\log(M)}{(\log(M-1))^Q}.
\end{equation*}
This gives the desired upper bound on $M$ and the claim follows.
\end{proof}

\section{Proof of Theorem \ref{main}}
We begin by proving the following theorem.

\begin{teor}\label{main theorem}
Let $Q>1$ and let $(X,d,\mu)$ be a minimally nice $Q$--upper regular $D$--quasiconvex space 
such that the measure $\mu$ is $Q$--regular on Whitney type balls. 
Suppose that $(X,k_d)$ is a $K$--roughly starlike Gromov hyperbolic space. 
Then $(X,d)$ satisfies both the Gehring--Hayman condition and the ball separation condition.
\end{teor}

\begin{proof}
Let us first prove that $(X,d)$ satisfies the Gehring--Hayman condition.
Because $(X,k_d)$ is Gromov hyperbolic and $K$--roughly starlike, $(X,\de,\mue)$ is uniform
 for a deformation as in Section 3 with respect to a base point $w\in X,$ 
where we choose $\epsilon \leq \min\{\epsilon_0,\frac{1}{8D},\frac{1}{2C^2}\},$  where
$\epsilon_0>0$ is 
as in the paragraph containing \eqref{G-H}, $C>1$ as in \eqref{constant C}. 
We know from \eqref{upper bound} and \eqref{lower bound} that the measure 
$\mue$ is $Q$--regular on Whitney type balls. 
We will consider $(X,d,\mu)$ as a conformal deformation of 
$(X,\de,\mue).$

Towards this end, define $\tilde{\rho}\colon (X,\de) \to (0,\infty)$ by setting
\[\tilde{\rho}(z)=\frac{d(z)}{\rho_\epsilon(z)}=(\sigma_\epsilon(z))^{-1}.\] 
First, let us prove that $\tilde{\rho}$ satisfies the Harnack inequality HI($A$) 
with some constant $A\geq 1.$ Let $z\in X$ and $x\in B_\epsilon(z,\epsilon\de(z)).$ 
By inequality \ref{WB1} and the triangle inequality we obtain that
\begin{equation}
\begin{split}
\tilde{\rho}(x)=\frac{d(x)}{\exp\{-\epsilon k_d(w,x)\}}&\leq \frac{(1+\epsilon C^2)d(z)}{\exp\{-\epsilon(k_d(w,z)+1)\}}\\
&\leq\exp\{\epsilon\}(1+\epsilon C^2)\tilde{\rho}(z),
\end{split}
\end{equation}
and 
\begin{equation}
\begin{split}
\tilde{\rho}(x)=\frac{d(x)}{\exp\{-\epsilon k_d(w,x)\}}&\geq \frac{(1-\epsilon C^2)d(z)}{\exp\{-\epsilon(k_d(w,z)-1)\}}\\
&\geq \frac{1-\epsilon C^2}{\exp\{\epsilon\}}\tilde{\rho}(z).
\end{split}
\end{equation}
Thus, for $A=\max\{\exp\{2\epsilon\}(1+\epsilon C^2)^2,\frac{\exp\{2\epsilon\}}{(1-\epsilon C^2)^2}\},$ the density $\tilde{\rho}$ 
satisfies 
\begin{equation}\label{HIA}
A^{-1}\leq \frac{\tilde{\rho}(x)}{\tilde{\rho}(y)}\leq A
\end{equation}
for all $x,y \in B_\epsilon(z,\epsilon \de(z))$ and each $z\in X.$

The density $\tilde{\rho}$ also satisfies the volume growth condition VG($B$) with the 
constant $B=C_\text{u}D^Q.$
Indeed, observe that
\begin{equation}\label{same}
\begin{split}
d_{\tilde{\rho}}(x,y)&=\inf_{\gamma_{xy}}\int_{\gamma_{xy}}\tilde{\rho}(z)\, \de s(z)\\
&=\inf_{\gamma_{xy}}\int_{\gamma_{xy}}(\sigma_\epsilon(z))^{-1}\sigma_\epsilon(z)\, ds(z)\\
&=\inf_{\gamma_{xy}} \ell_d(\gamma_{xy}),
\end{split}
\end{equation}
where the infimum is taken over all curves $\gamma_{xy}$ joining points $x$ and $y.$ 
Since $(X,d)$ is $D$-quasiconvex, it follows that $(X,d_{\tilde{\rho}})$ is bi--Lipschitz 
equivalent to $(X,d)$ and furthermore, 
we have that $B_{\tilde{\rho}}(z,r)\subset B_d(z,Dr)$ for all $z\in X$ and $r>0.$ 
Thus from the $Q$--upper regularity condition \eqref{upper regular} it follows that
\begin{equation}
\begin{split}
\mu_{\tilde{\rho}}(B_{\tilde{\rho}}(z,r))&=\int_{B_{\tilde{\rho}}(z,r)}\tilde{\rho}(x)^Q\, d\mue(x)\\
&=\int_{B_{\tilde{\rho}}(z,r)}(\sigma_\epsilon(x))^{-Q}(\sigma_\epsilon(x))^Q\, d\mu(x)\\
&=\mu(B_{\tilde{\rho}}(z,r))\\
&\leq\mu(B_d(z,Dr))\\&\leq C_\text{u}D^Qr^Q
\end{split}
\end{equation}
for every $z\in X$ and $r>0.$

Hence, $\tilde{\rho}$ is a conformal density. The corresponding deformation of the 
metric space $(X,\de)$ 
with $\tilde{\rho}$ results in an inner metric space $(X,\ell_d)$ which is bi--Lipschitz 
equivalent to the original metric space $(X,d).$ 
We aim to apply \cite[Theorem 1.1]{KL} that gives a Gehring-Hayman condition for 
conformal deformations of certain uniform spaces. As stated in \cite{KL}, 
this theorem applies in
our setting to quasihyperbolic geodesics with respect to the metric $d_\epsilon,$ but not
directly to the geodesics $[x,y]_{k_d}.$ However, the proof in \cite{KL} gives the
estimate
 \begin{equation*}
\ell_d(\beta_{xy})\leq C_{\text{gh}}\ell_d(\gamma_{xy})
\end{equation*}
for each $D_\epsilon$--uniform curve $\beta_{xy}$, with $C_{\text{gh}}$ depending only on $D_\epsilon$ and the 
data associated
to our conformal deformation and $(X,d_{\epsilon},\mu_\epsilon).$ Recalling that each
quasihyperbolic geodesic $[x,y]_{k_d}$ is a $D_{\epsilon}$--uniform curve in $(X,d_{\epsilon})$
we conclude with the desired Gehring--Hayman condition.

Let us then prove that $(X,d)$ satisfies the ball separation condition. Let $x,y \in X,$ $u\in [x,y]_{k_d}$ and $\gamma_{xy} \subset X$ be a curve joining $x$ and $y.$ Let $(X,\de,\mue)$ be the deformation of $(X,d,\mu)$ as before. We may assume that $\ell_\epsilon([x,u]_{k_d})\leq \ell_\epsilon([u,y]_{k_d}).$ Because $[x,y]_{k_d}$ is a $D_\epsilon $--uniform curve in $(X,\de),$ we have that
\begin{equation*}
\dist_\epsilon(u,\gamma_{xy})\leq \ell_\epsilon([x,u]_{k_d})\leq D_\epsilon\de(u),
\end{equation*}
and
\begin{equation*}
\dist_\epsilon(u,\gamma_{xy})\leq \ell_\epsilon([x,u]_{k_d}) 
\leq D_\epsilon\diam_\epsilon(\gamma_{xy}).
\end{equation*}
Thus assumptions of Lemma \ref{main lemma} hold,
and hence there is a constant $C_{\text{bs}}\geq 1,$ depending on $\epsilon$ and the hypotheses, such that
\[\gamma\cap B_d(u,C_{\text{bs}}d(u))\neq \emptyset.\]
\end{proof}

Balogh and Buckley proved in \cite[Theorem 2.4 and Theorem 6.1]{BB} that, for a 
minimally nice length space $(X,d)$ that satisfies both the Gehring--Hayman condition and the
ball separation condition, the associated space $(X,k_d)$ is Gromov hyperbolic. 
Therefore we have the following corollary to Theorem \ref{main theorem}.

\begin{seur}\label{coro}
Let $Q>1$ and let $(X,d,\mu)$ be a minimally nice $Q$--upper regular length space such 
that the measure $\mu$ is $Q$--regular on Whitney type balls. Suppose that $(X,k_d)$ is 
$K$--roughly starlike. Then the quasihyperbolic space $(X,k_d)$ is Gromov hyperbolic if 
and only if $(X,d)$ satisfies both the Gehring--Hayman condition and the ball separation 
condition.
\end{seur}

Now we are able to deduce Theorem \ref{main}.\\\\
\begin{proof1}
From \cite[Theorem 3.1]{BB} it follows that $(\Omega,k)$ is $K$--roughly starlike, 
because $(X,d)$ is annularly quasiconvex and $\Omega \subset X$ is a bounded and proper 
subdomain. Hence the claim follows from Corollary \ref{coro}.
\end{proof1}
~\\

\renewcommand\refname{References}

\end{document}